\documentclass[12pt,reqno]{amsart}

\usepackage{amsfonts}
\usepackage{amscd}
\usepackage{amssymb}
\usepackage{amsfonts}
\usepackage{amscd}
 \usepackage{amsmath} 
\usepackage{mathrsfs}
\usepackage{enumerate}
\usepackage{float} 
\usepackage[pdftex]{graphicx}
\allowdisplaybreaks

\usepackage{hyperref}

\date{\today}

\newtheorem{thm}{Theorem}[section]
\newtheorem{cor}[thm]{Corollary}
\newtheorem{lem}[thm]{Lemma}

\newtheorem{prop}[thm]{Proposition}
\theoremstyle{definition}

\theoremstyle{remark}
\newtheorem{rem}[thm]{Remark}

\numberwithin{equation}{section}

\newcommand{\R}{\mathbb R}

\newcommand{\Na}{\mathbb N}

\newcommand{\C}{{\mathbb C}}

\renewcommand{\Re}{\operatorname{Re}}
\renewcommand{\Im}{\operatorname{Im}}
 
\usepackage{color}

\newcommand{\red}[1]{\textcolor{red}{#1}}

\title[Chernoff's theorem for symmetric spaces]
{On a theorem of Chernoff on rank one \\ Riemannian symmetric spaces  }

\author[ Ganguly, Manna and Thangavelu]{ Pritam Ganguly, Ramesh Manna and Sundaram Thangavelu}

\address{Department of Mathematics\\
Indian Institute of Science\\
560 012 Bangalore, India}

\email{pritamg@iisc.ac.in, rameshmanna@iisc.ac.in, veluma@iisc.ac.in}

 \date{}
 
\begin{document}

\maketitle

\begin{abstract} 
 In 1975, P.R. Chernoff used iterates of the Laplacian on $\mathbb{R}^n$ to prove an $L^2$ version of the Denjoy-Carleman theorem which provides a sufficient condition for a smooth function on $\mathbb{R}^n$ to be quasi-analytic. In this paper we prove an exact analogue of Chernoff's theorem for all  rank one Riemannian symmetric spaces (of noncompact and compact types) using  iterates of the associated Laplace-Beltrami operators. 
 \end{abstract}

\section{Introduction and the main results}
 The paramount property of an analytic function is that it is  completely determined by its value and the values of all its derivatives at a single point. Borel first perceived that there is a more larger class of smooth functions than that of analytic functions which has this magnificent property. He coined the term \textit{quasi-analytic} for such class of functions. In exact terms a subset of smooth functions on an interval $(a,b)$ is called a quasi-analytic class if for any function $f$ from that set and $x_0\in (a,b)$, $\frac{d^n}{dx^n}f(x_0)=0$ for all $n\in\mathbb{N}$ implies $f=0.$ Now recall that a smooth function on an interval $I$ is analytic provided its Taylor series converges to the function on $I$ which naturally restricts the growth of derivatives of that function. In fact, if for every $n$, $\|\frac{d^n}{dx^n}f\|_{L^{\infty}(I)}\leq C n!A^n$ for some constant $A$ depending on $f$ then the Taylor series of $f$ converges to $f$ uniformly and the converse is also true. This drives an analytic mind to investigate whether relaxing growth condition on the derivatives generates quasi-analytic class. In 1912 Hadamard proposed the problem of finding  sequence $\{M_n\}_n$ of positive numbers such that the class $C\{M_n\}$ of smooth functions on $I$ satisfying  $\|\frac{d^n}{dx^n}f\|_{L^{\infty}(I)}\leq A_f^nM_n$  for all $f\in C\{M_n\}$ is a quasi-analytic class. A solution to this problem is provided by a theorem of Denjoy and Carleman where they showed that $C\{M_n\}$ is quasi-analytic if and only if $\sum_{n=1}^{\infty}M_n^{-1/n}=\infty$. As a matter of fact Denjoy \cite{Den} first proved a sufficient condition  and later Carleman \cite{Car} completed the theorem giving a necessary and sufficient condition.  A short proof of this theorem based on complex analytic ideas can be found in Rudin \cite{R}. A several variable analogue of this theorem has been obtained by Bochner and Taylor \cite{BT} in 1939.

Later in 1950, instead of using all partial derivatives, Bochner used iterates of the Laplacian $\Delta$ and proved an analogue of Denjoy-Carleman theorem which reads as follows: if $f\in C^{\infty}(\mathbb{R}^n)$ satisfies $\sum_{m=1}^{\infty}\|\Delta^mf\|^{-1/m}_{\infty}=\infty,$ then the condition  $\Delta^mf(x)=0$ for all $m\geq 0$ and for all $x$ in a set $U$ of analytic determination implies $f=0.$  Building upon the works of  Masson- McClary \cite{MM} and Nussbaum \cite{N}, in 1972  Chernoff \cite{C1} used operator theoretic arguments to study quasi-analytic vectors. As an application he improved the above mentioned result of Bochner by proving the following very interesting result in 1975.
 
\begin{thm}\cite[Chernoff]{C}
		Let $f$ be a smooth function on $\mathbb{R}^n.$ Assume that $\Delta^mf\in L^2(\mathbb{R}^n)$ for all $m\in \mathbb{N}$ and $\sum_{m=1}^{\infty}\|\Delta^mf\|_2^{-\frac{1}{2m}}=\infty.$ If $f$ and all its partial derivatives   vanish at  a point  $ a \in \R^n$, then $f$ is identically zero.
\end{thm}

In this paper we prove an analogue of Chernoff's theorem for the Laplace-Beltrami operator on rank one symmetric spaces of both compact and  noncompact types. In order to state our results we first need to introduce some notations. Let $G$ be a connected, noncompact semisimple Lie group with finite centre and $K$  a maximal compact subgroup of $G$. Let $ X=G/K$ be the associated symmetric space which is  assumed to have  rank one.   The origin  $o$ in the symmetric space is given by the identity coset $eK$  where $e$ is the identity element in $G$. We know that $X$ is a Riemannian manifold equipped with a $G$ invariant  metric on it.  We denote by  $\Delta_{X}$  the Laplace-Beltrami operator associated to $X$. 

 The Iwasawa decomposition of $G$ reads as $G=KAN$ where $A$ is abelian and $N$ is a nilpotent Lie group. Let $\mathfrak{ g}$ and $\mathfrak{a}$ stand for  the Lie algebras corresponding to $G$ and $A$ respectively. Here $\mathfrak{a}$ is one dimensional since $X$ is of rank one. It is well known that every element of $\mathfrak{ g}$  gives rise to a left invariant vector field on $G$. Let $H$ be the left invariant vector field corresponding to  a fixed  basis element of $\mathfrak{a}.$ We will describe all these notations in detail in the next section.  As an exact analogue of Chernoff's theorem for $ X $ we prove the following:

\begin{thm}
	\label{C}
 Let ${X}=G/K$ be a rank one symmetric space of noncompact type. Suppose $f\in C^{\infty}(X)$ satisfies $\Delta_X^mf\in L^2(X)$ for all $m\geq 0$ and $\sum_{m=1}^{\infty}\|\Delta_{X}^mf\|_2^{-\frac{1}{2m}}=\infty.$ If   ${H}^lf(eK)=0$ for all $l\geq 0$ then $f$ is identically zero.
\end{thm} 
As an immediate consequence of the above result we obtain an analogue of the $L^2$ version of the classical Denjoy-Carleman theorem using iterates of the Laplace-Beltrami operator on $X = G/K$.
 \begin{cor}
 Let ${X}=G/K$ be a rank one symmetric space of noncompact type.	Let $\{M_k\}_k$ be a log convex sequence. Define $\mathcal{C}(\{M_k\}_k,\Delta_X,X)$ to be the class of all smooth functions $f$ on $X$ satisfying   $\Delta_X^mf\in L^2(X)$ for all $  m\in \mathbb{N}$ and $\|\Delta_X^kf\|_2\leq M_k\lambda(f)^k$ for some constant $\lambda(f)$ depending  on $f$. Suppose that $\sum_{k=1}^{\infty}M_k^{-\frac{1}{2k}}=\infty.$ Then every member of that class is quasi-analytic.
 \end{cor}

As Chernoff's theorem is a useful tool in establishing uncertainty principles of Ingham's type, proving analogues of  Theorem 1.1 in contexts other than Euclidean spaces have received considerable attention in recent years. Recently, an analogue of Chernoff's theorem for the sublaplacian on the Heisenberg group has been proved in \cite{BGST}. For noncompact Riemannian symmetric spaces $ X = G/K$,  without any restriction on the rank, the following weaker version of Theorem 1.2 has been proved in Bhowmik-Pusti-Ray \cite{BPR}.

\begin{thm}[Bhowmik-Pusti-Ray]
	\label{thm-BPR}
 Let ${X}=G/K$ be a  noncompact Riemannian symmetric space and let $ \Delta_X $ be the associated Laplace-Beltrami operator. Suppose $f\in C^{\infty}(X)$ satisfies $\Delta_X^mf\in L^2(X)$ for all $m\geq 0$ and $\sum_{m=1}^{\infty}\|\Delta_{X}^mf\|_2^{-\frac{1}{2m}}=\infty.$ If $ f $ vanishes on a non empty open set,  then $f$ is identically zero.
\end{thm} 

In proving the above theorem, the authors have made use of a result of de Jeu \cite{J}.  In the case of rank one symmetric spaces, a different proof was given by the first and third authors of this article by making use of spherical means and an analogue of Chernoff's theorem for the Jacobi transform proved in \cite{GT2}. In fact, we only need to use the one dimensional version of de Jeu's theorem which is equivalent to the Denjoy-Carleman theorem. Our proof of Theorem 1.2 is built upon the ideas used in \cite{GT2}. In a very recent preprint, Bhowmik-Pust-Ray have proved the following improvement of their Theorem \ref{thm-BPR}. In what follows let $D(G/K) $ denote the algebra of differential operators on $ G/K $ which are invariant under  the  (left) action of $ G.$

\begin{thm}[Bhowmik-Pusti-Ray]
	\label{thm-BPR1}
 Let ${X}=G/K$ be a  noncompact Riemannian symmetric space and let $ \Delta_X $ be the associated Laplace-Beltrami operator. Suppose $f\in C^{\infty}(G/K)$ be a left $ K $-invariant function on $ X $ which satisfies $\Delta_X^mf\in L^2(X)$ for all $m\geq 0$ and  $\sum_{m=1}^{\infty}\|\Delta_{X}^mf\|_2^{-\frac{1}{2m}}=\infty.$  If there is an $ x_0 \in X$  such that $ Df(x_0) $ vanishes for all $ D \in D(G/K)$ then $f$ is identically zero.
\end{thm} 
 \begin{rem}  Observe that in the above theorem the function $ f $ is assumed to be $ K$-biinvariant. The problem of proving the same for all functions on $ X $ is still open. However, in the case of rank one symmetric spaces we have proved Theorem 1.2 for all functions $ f $. Moreover, we only require that $ H^lf(eK) =0 $ for all $ l \geq  0.$ Here we can also take any $x_0\in X$ in place of $eK$ using translation invariance of Laplacian and $H.$
\end{rem}

 We remark that  the condition $ H^lf(eK) =0 $  is the counterpart of  $ (\frac{d}{dr})^kf(r\omega)|_{r=0} = 0 $ where $ x = r\omega, r> 0, \omega \in \mathbb{S}^{n-1}  $ is the polar decomposition of $ x \in \R^n.$ Indeed, as can be easily checked
$$  \left(\frac{d}{dr}\right)^kf(r\omega)  = \sum_{|\alpha|=k} \partial^\alpha f(r\omega)\,  \omega^\alpha $$
and hence $(\frac{d}{dr})^kf(r\omega)|_{r=0} = 0 $ for all $ k $ if and only if $ \partial^\alpha f(0) = 0 $ for all $ \alpha.$ This observation plays an important role in formulating the right analogue Chernoff's theorem for compact Riemannian symmetric spaces. In view of the above observation, Chernoff's theorem for the Laplacian on $ \R^n $ can be stated in the following form.

\begin{thm} Let $f$ be a smooth function on $\mathbb{R}^n.$ Assume that $\Delta^mf\in L^2(\mathbb{R}^n)$ for all $m\in \mathbb{N}$ and $\sum_{m=1}^{\infty}\|\Delta^mf\|_2^{-\frac{1}{2m}}=\infty.$ If  $(\frac{d}{dr})^kf(r\omega)|_{r=0} = 0 $ for all $ k $ and $ \omega \in \mathbb{S}^{n-1} ,$ then $f$ is identically zero.
\end{thm}
We can give a proof of the above theorem by reducing it to a theorem for Bessel operators. Recall that written in polar coordinates the  Laplacian takes the form
\begin{equation}\label{polar-Lap}  \Delta = \frac{\partial^2}{\partial r^2}+ \frac{n-1}{r} \frac{\partial}{\partial r}+\frac{1}{r^2} \Delta_{\mathbb{S}^{n-1}} \end{equation}
where $ \Delta_{\mathbb{S}^{n-1}} $ is the spherical Laplacian on the unit sphere $ \mathbb{S}^{n-1}.$ By expanding the function $ F(r,\omega) = f(r\omega) $ in terms of spherical harmonics on $ \mathbb{S}^{n-1} $ and making use of Hecke-Bochner formula, we can easily reduce Theorem 1.7 to a sequence of theorems for the Bessel operator $ \partial_r^2+ (n+2m+1)r^{-1} \partial_r $ for various values of $ m \in \mathbb{N}.$ This idea has been already used in the paper \cite{GT2}. A similar expansion in the case of noncompact Riemannian symmetric spaces leads to Jacobi operators as done in \cite{GT2} which will be used in proving Theorem 1.2. As the proof of the above theorem is similar to and easier  than that of Theorem 1.2, we will not present it here.

\begin{rem} We remark in passing that the above theorem can also  be proved in the context of Dunkl  Laplacian on $ \R^n $ associated to root systems. We would also like to mention that analogues of Chernoff's theorem can be proved for the Hermite operator $ H $ on $ \R^n$ and the special Hermite operator $ L $ on $ \C^n.$ Again the idea is to make use of Hecke-Bochner formula for the Hermite  and special Hermite projections (associated to their spectral decompositions).
\end{rem}

  So far we have only considered non compact Riemannian symmteric spaces, but now we turn our attention to proving an analogue of Theorem 1.2 for compact, rank one symmetric spaces.  We make use of the well known classification of such spaces in formulating and proving a Chernoff theorem for the Laplace-Beltrami operator. It turns out that we only need to prove such a result for the spherical Laplacian on spheres in Euclidean spaces.  

Let $(U,K)$ be a compact symmetric pair and $S= U/K$ be the associated symmetric space. Here  $U$ is a compact semisimple Lie group and $K$ is a connected subgroup of $U$. We assume that $S$ has rank one.   Being a compact Riemannian manifold, $S$ admits a Laplace-Beltrami operator $\tilde{\Delta}_S. $ It is customary to add a suitable constant $\rho_S$ and work with $\Delta_S= -\tilde{\Delta}_S+\rho_S^2.$ This way we can arrange that $\Delta_S\geq \rho_S^2>0.$  In \cite{W} H.C.Wang  has completely classified all rank one compact symmetric spaces. To be more precise, $S$ is one of the followings: The unit sphere $\mathbb{S}^q= SO(q+1)/SO(q)$, the real projective space $P_q(\mathbb{R})=SO(q+1)/O(q)$, the complex projective space $P_l(\mathbb{C})$, the quaternion projective space $P_l(\mathbb{H})$ and the Cayley projective space $P_2(\mathbb{C}ay)= F_4/Spin(9).$ In each case, $S$ comes up with an appropriate polar form $(0,\pi)\times \mathbb{S}^{k_S}$ where $k_S$ depends on the symmetric space $S$. As a consequence, functions on $S$ can be identified with functions on the product space $ Y =(0,\pi)\times \mathbb{S}^{k_S},$  see Section 4 for more details.   We prove the following analogue of Chernoff's theorem:

\begin{thm}
	Let $S$ be a rank one Riemannian symmetric space of compact type. Suppose $f\in C^{\infty}(S)$ satisfies $\Delta_{S}^mf\in L^2(S)$ for all $m\geq 0$ and $\sum_{m=1}^{\infty}\|\Delta_{S}^mf\|_2^{-\frac{1}{2m}}=\infty.$ If the function $F$ on $(0,\pi)\times \mathbb{S}^{k_S}$ associated to $ f $ on $S$  satisfies   $ \frac{\partial^m}{\partial\theta^m}\big|_{\theta=0}F(\theta,\xi)=0$ for all $m\geq 0$, then $f$ is identically zero.
\end{thm}  

  In the context of Theorem 1.7, by identifying $ \R^n $ with $ (0,\infty) \times \mathbb{S}^{n-1} $ every function $ f $ on $ \R^n $ gives rise to a function $ F(r,\omega) $ on $ (0,\infty) \times \mathbb{S}^{n-1} $ and  in view of \ref{polar-Lap}, the action of $ \Delta $ on $ f $ takes the form,
$$ \Delta f(r,\omega) = \frac{\partial^2}{\partial r^2} F(r,\omega)+ \frac{n-1}{r} \frac{\partial}{\partial r}F(r,\omega)+\frac{1}{r^2} \Delta_{\mathbb{S}^{n-1}}F(r,\omega).$$
There is a similar decomposition of $ \Delta_S $ as a sum of a Jacobi operator on $ (0,\pi) $ and the spherical Laplacian $ \Delta_{\mathbb{S}^{k_S}}$ and this justifies our formulation of Theorem 1.9.

We complete this introduction with a brief description of the plan of the paper. In Section 2 we recall the requisite preliminaries on noncompact Riemannian symmetric spaces and in Section 3 we prove our version of  Chernoff's theorem  for the Laplace-Beltrami operator. In Section 4, after recalling necessary results from the theory of compact symmetric spaces and setting up the notations, we prove Theorem 1.9. We refer the reader to the papers \cite{GT1} and \cite{GT2} for related ideas.

\section{Preliminaries on Riemannian symmetric spaces of non-compact type}
 In this section we describe the relevant theory regrading the  harmonic analysis on rank one Riemannian symmetric spaces of noncompact type. General references for this section are the monographs of Helgason \cite{H1} and \cite{H2}. 

Let $G$ be a connected, noncompact semisimple Lie group with finite centre. Suppose $\mathfrak{ g}$ denotes its Lie algebra. With respect to a fixed Cartan involution $\theta$ on $\mathfrak{ g}$ we have the decomposition  $\mathfrak{ g}=\mathfrak{k}\oplus \mathfrak{p} .$ Here $\mathfrak{k}$ and $\mathfrak{p}$ are the $+1$ and $-1$ eigenspaces of $\theta$ respectively. Let $\mathfrak{a}$ be the maximal abelian subspace of $\mathfrak{p}$. Also assume that the dimension of $\mathfrak{a}$ is one. Now we know that the involution $\theta$ induces an automorphism $\Theta$ on $G$ and $K=\{g\in G:\Theta(g)=g\}$ is a maximal compact subgroup of $G$. We consider the homogeneous space ${X}=G/K$ which a is a smooth manifold endowed with a $G$-Riemannian metric induced by the restriction of the Killing form $\mathfrak{B}$ of $\mathfrak{g}$ on $\mathfrak{p}$. This turns $X$  into  a rank one Riemannian symmetric space of noncompact type and every such space can be realised this way.  

Let $\mathfrak{a}^*$ denote the dual of $\mathfrak{a}$. Given $\alpha\in \mathfrak{a}^*$ we define 
$$\mathfrak{ g}_{\alpha}:=\{X\in \mathfrak{ g}:[Y,X]=\alpha(Y)X, \forall\ Y\in \mathfrak{a}\}.$$
 Now  $\Sigma:=\{\alpha\in \mathfrak{a}^*: \mathfrak{ g}_{\alpha}\neq \{0\}\}$ is the set of all resticted roots of the pair $(\mathfrak{g},\mathfrak{a})$. Let $\Sigma_{+}$ denote the set of all positive roots with respect to a fixed Weyl chamber. It is known that $\mathfrak{n}:=\oplus_{\alpha\in \Sigma_{+}}\mathfrak{g}_{\alpha}$ is a nilpotent subalgebra of $\mathfrak{g}$ and we have the Iwasawa decomposition $\mathfrak{ g}=\mathfrak{k}\oplus\mathfrak{a}\oplus\mathfrak{n}.$ Now writing $N=\exp \mathfrak{n}$ and $A=\exp \mathfrak{a}$ we obtain $G=KAN$ where $A$ is abelian and $N$ is a nilpotent subgroup of $G$. Moreover, $A$ normalizes $N$. In view of this decomposition every $g\in G$ can be uniquely written as $g=k(g)\ \exp H(g) n(g)$ where  $H(g)$ belongs to $\mathfrak{a}$. Also we have $G=NAK$ and with respect to this decomposition  we write $g\in N\exp A(g) K$ where the functions $A$ and $H$ are related via $A(g)=-H(g^{-1}). $  Now in the rank one case when dimension of $\mathfrak{a}$ is one, $\Sigma$ is given by either  $\{\pm\gamma\} $ or $\{\pm \gamma, \pm 2\gamma\}$ where $\gamma$ belongs to $\Sigma_{+}$.  Let $\rho:=(m_{\gamma}+m_{2\gamma})/2$ where $m_{\gamma}$ and $m_{2\gamma}$ denote the multiplicities of the roots $\gamma$ and $2\gamma$ respectively. The Haar measure $dg$ on $G$ is given by 
 $$\int_{G}f(g)dg=\int_K\int_{A}\int_N f(k a_t n)e^{2\rho t}dkdtdn. $$ 
 The measure $dx$ on $X$ is induced from the Haar measure $dg$ via the relation 
 $$\int_Gf(gK)dg=\int_X f(x)dx.$$  
 Suppose $M$ denotes the centralizer of $A$ in $K$. The polar decomposition of $G$ reads as $G=KAK$ in view of which we can write each $g\in G$ as $g=k_1a_rk_2$ with $k_1,k_2\in K$. Actually the map $(k_1,a_r,k_2)\rightarrow k_1a_rk_2$ of $K\times A\times K$ into $G$ induces a diffeomorphism of $K/M\times A_{+}\times K$ onto an open dense subset of $G$ where $A_{+}=\exp \mathfrak{a}_{+}$ and $\mathfrak{a}_{+}$ is the fixed positive Weyl chamber which basically can be identified with $(0,\infty)$ in our case.

 It is also well-known that each $\bf{X}\in \mathfrak{ g}$ gives rise to a left invariant vector field on $G$ by the prescription 
 $${\bf{X}}f(g)=\frac{d}{dt}\bigg|_{t=0}f(g.\exp(t{\bf{X}})),~g\in G.$$ Since $\mathfrak{a}$ is one dimensional, we fix a basis $\{H\}$ of $\mathfrak{a}$. By an abuse of notation, we denote the left invariant vector field corresponding to this basis element by $H$. Infact, we can write  $A=\{a_r=\exp(rH):r\in\mathbb{R}\}.$
 
\subsection{Helgason Fourier transform}Define the function $A:X\times K/M\rightarrow \mathfrak{a}$   by $A(gK,kM)=A(k^{-1}g).$ Note that $A$ is right $K$-invariant in $g$ and right $M$-invariant in $K$.  In what follows we denote the elements of $X$ and $K/M$ by $x$ and $b$ respectively. Let $\mathfrak{a}^*$ denote the dual of $\mathfrak{a}$ and $\mathfrak{a}^*_{\C}$ be its complexification. Here in our case $\mathfrak{a}^*$ and $\mathfrak{a}^*_{\C}$ can be identified with $\R$ and $\C$ respectively. For each $\lambda\in \mathfrak{a}^*_{\C}$ and $b\in K/M$, the function $x\rightarrow e ^{(i\lambda+\rho)A(x,b)}$ is a joint eigenfunction of all invariant differential operators on $X.$ For $f\in C^{\infty}_c(X)$, its Helgason Fourier transform is a function $\widetilde{f}$ on $\mathfrak{a}^*_{\C}\times K/M$ defined by 
$$\tilde{f}(\lambda, b)= \int_X f(x)e ^{(-i\lambda+\rho)A(x,b)}dx,~ \lambda\in \mathfrak{a}^*_{\C},~ b\in K/M . $$ Moreover, we know that if $f\in L^1(X)$ then $\widetilde{f}(.,b)$ is a continuous function on $\mathfrak{a}^*$ which extends holomorphically to a domain containing $\mathfrak{a}^*.$ The inversion formula for $f\in C^{\infty}_c(X)$ says that 
$$f(x)=c_{X}\int_{-\infty}^{\infty}\int_{K/M}\widetilde{f}(\lambda,b)e ^{(i\lambda+\rho)A(x,b)}|c(\lambda)|^{-2}dbd\lambda$$ where $d\lambda$ stands for usual Lebesgue measure on $\R$ (i.e., $\mathfrak{a}^*$) , $db$ is the normalised measure on $K/M$ and $c(\lambda)$ is the Harish-Chandra $c$-function. The constant $c_X$ appearing in the above formula is explicit and depends on the symmetric space $X$ (See e.g., \cite{H2}). Also for $f\in L^1(X)$ with $\widetilde{f}\in L^1(\mathfrak{a}^*\times K/M,|c(\lambda)|^{-2}dbd\lambda)$, the above inversion formula holds for a.e. $x\in X.$ Furthermore, the mapping $f\rightarrow \widetilde{f}$ extends as an isometry of $L^2(X)$ onto $L^2(\mathfrak{a}^*_{+}\times K/M, |c(\lambda)|^{-2}d\lambda db) $ which is known as the Plancherel theorem for the Helgason Fourier transform.

We also need to use certain irreducible representations of $K$ with $M$-fixed vectors. Suppose $\widehat{K_0}$ denotes the set of all irreducible unitary representations of $K$ with $M$ fixed vectors. Let $\delta\in \widehat{K_0}$ and $V_{\delta}$ be the finite dimensional vectors space on which $\delta $ is realised. We know that $V_{\delta}$ contains a unique normalised $M$-fixed vector $ v_1$ (See Kostant \cite{Ks}). Consider  an orthonormal basis $\{v_1,v_2,...,v_{d_{\delta}}\}$ for $V_{\delta}$. For $\delta\in \widehat{K_0}$ and $1\le j\le d_{\delta}$, we define  $$Y_{\delta,j}(kM)=(v_j, \delta(k)v_1),~ kM\in K/M.$$ It can be easily checked that $Y_{\delta,1}(eK)=1$ and  moreover, $Y_{\delta,1}$ is $M$-invariant. 
\begin{prop}[\cite{H2}]
	\label{H2}
	The set $\{Y_{\delta,j}:1\le j\le d_{\delta},\delta\in \widehat{K_0}\}$ forms an orthonormal basis for $L^2(K/M)$.  
\end{prop} 
We can get an explicit realisation of $\widehat{K_0}$ by identifying $K/M$ with the unit sphere in $\mathfrak{p}$. By letting $\mathcal{H}^m$ to stand for the space of homogeneous harmonic polynomials of degree $m$ restricted to the unit sphere, we have the following spherical harmonic decomposition
$$L^2(K/M)=\displaystyle\oplus_{m=0}^{\infty}\mathcal{H}^m .$$  Thus the functions $Y_{\delta,j}$ can be identified with the spherical harmonics. 

Given $\delta\in \widehat{K_0}$ and $\lambda\in \mathfrak{a}^*_{\C}\,  (i.e., \C~ \text{in our case})$ we consider the spherical functions of type $\delta$ defined by 
$$\Phi_{\lambda,\delta}(x):=\int_K e^{(i\lambda+\rho)A(x,kM)}Y_{\delta,1}(kM)dk.$$ These are eigenfunctions of the Laplace-Beltrami operator $\Delta_X$ with eigenvalue $-{(\lambda^2+\rho^2)}.$ When $\delta$ is the trivial representation for which  $Y_{\delta,1}=1,$  the function $\Phi_{\lambda,\delta}$ is called  the elementary spherical function, denoted by $\Phi_\lambda$. More precisely, 
$$\Phi_\lambda(x)=\int_K e^{(i\lambda+\rho)A(x,kM)} dk.$$ Note that these functions are $K$-biinvariant. The spherical functions can be expressed in terms of Jacobi functions. In fact, if $x=gK$ and $g=ka_rk^{'}$ (polar decomposition), $\Phi_{\lambda,\delta}(x)=\Phi_{\lambda,\delta}(a_r)$. Suppose
\begin{equation}
\label{ab}
\alpha=\frac12(m_{\gamma}+m_{2\gamma}-1),~ \beta=\frac12(m_{2\gamma }-1).
\end{equation}   For each $\delta\in \widehat{K_0}$ there exists a pair of integers $(p,q)$ such that 
\begin{equation}
\Phi_{\lambda,\delta}(x)=Q_{\delta}(i\lambda+\rho)(\alpha+1)_p^{-1}(\sinh r)^p(\cosh r)^q \varphi_{\lambda}^{(\alpha+p,\beta+q)}(r)
\end{equation} where $\varphi_{\lambda}^{(\alpha+p,\beta+q)}$ are the Jacobi functions of type $(\alpha+p,\beta+q)$ and $Q_{
	\delta}$ are the Kostant polynomials given by
\begin{equation}
\label{kos}
Q_{\delta}(i\lambda+\rho)=\left(\frac{1}{2}(\alpha+\beta+1+i\lambda)\right)_{(p+q)/2}\left(\frac{1}{2}(\alpha-\beta+1+i\lambda)\right)_{(p-q)/2}.
\end{equation} 
 In the above we have used the notation $(z)_m=z(z+1)(z+2)...(z+m-1).$  The following result proved in Helgason \cite{H2} will be very useful for our purpose:
\begin{prop}
	\label{fsp}
	Let $\delta\in \widehat{K_0}$ and $1\le j\le d_{\delta}$. Then we have 
	\begin{equation}
	\int_K e^{(i\lambda+\rho)A(x,k^{'}M)}Y_{\delta,j}(k^{'}M)dk^{'}=Y_{\delta,j}(kM)\Phi_{\lambda,\delta}(a_r),~x=ka_r\in X. 
	\end{equation}
\end{prop}
 
 We refer the reader to the papers \cite{Jo} and \cite{JW} for all the results recalled in this subsection.
 
 \subsection{Spherical Fourier transform}
 We say that a function $f$ on $G$ is $K$-biinvariant if $f(k_1gk_2)=f(g)$ for all $k_1,k_2\in K$. It can be checked that if $f$ is a $K$-biinvariant integrable function then its Helgason Fourier transform 
 $\widetilde{f}(\lambda,b)$ is independent of $b\in K/M$ and by a little abuse of notation we write this as  
 $$\tilde{f}(\lambda)=\int_{X}f(x)\Phi_{-\lambda}(x)dx.$$ This is called the spherical Fourier transform. Now since $f$ is $K$ biinvarinat, using the polar decomposition $g=k_1a_rk_2$, we can view $f$ as a function on $A$ alone: $f(g)=f(a_r)$. So the above integral takes the following polar form:
 $$\tilde{f}(\lambda)=\int_{ 0}^{\infty}f(a_r)\varphi_{\lambda}(r) w_{\alpha,\beta}(r)dr$$ where $ w_{\alpha,\beta}(r)=(2\sinh r)^{2\alpha+1}(2\cosh r)^{2\beta+1}$ and $\Phi_{-\lambda}(a_r)=\varphi_{\lambda}(r)$ are given by Jacobi function $\varphi_{\lambda}^{\alpha,\beta}(r)$ of type $(\alpha,\beta).$ Here $\alpha$ and $\beta$ are associated to the symmetric space as mentioned above. So it is clear that the spherical Fourier transform is basically Jacobi transform of type $(\alpha,\beta)$. In the rest of the section we describe certain results from the theory of Jacobi analysis. 
 
 Let $\alpha,\beta,\lambda\in \mathbb{C}$ and $-\alpha\notin \mathbb{N}.$ The Jacobi functions $\varphi_{\lambda}^{(\alpha, \beta)}(r) $ of type $(\alpha,\beta)$ are  solutions of 
 the initial value problem
 \begin{align*}
 (\mathcal{L}_{\alpha,\beta}+ \lambda^2+ \varrho^2)\varphi_{\lambda}^{( \alpha ,\beta)}(r) =0,\,\,\, \varphi_{\lambda}^{( \alpha , \beta )}(0)=1 
 \end{align*} where $\mathcal{L}_{\alpha,\beta}$ is  the Jacobi operator defined by 
 $$\mathcal{L}_{\alpha,\beta}:=\frac{d^2}{dr^2}+((2\alpha+1)\coth r+(2\beta+1)\tanh r) 
 \frac{d}{dr}$$ and $\varrho=\alpha+\beta+1.$ Thus Jacobi functions $\varphi_{\lambda}^{(\alpha,\beta)}$ are eigenfunctions of $\mathcal{L}_{\alpha,\beta}$ with eigenvalues $-(\lambda^2+\varrho^2).$ These are even functions on $ \R $ and are expressible in terms of hypergeometric functions. For certain values of the parameters $ (\alpha, \beta) $ these functions arise naturally as spherical functions on Riemannian symmetric spaces of noncompact type.    The Jacobi transform of a suitable function $f$ on $\R^+$ is defined by 
 $$J_{\alpha,\beta}f(\lambda)=\int_{ 0}^{\infty}f(r)\varphi_{\lambda}^{(\alpha,\beta)}(r) {w}_{\alpha,\beta}(r)dr.$$ This is also called the Fourier-Jacobi transform of type $(\alpha,\beta).$  It can be checked that the operator $\mathcal{L}_{\alpha,\beta}$ is selfadjoint on $L^2(\R^+, {w}_{\alpha,\beta}(r)dr)$ and that
 $$\widetilde{\mathcal{L}_{\alpha,\beta}f}(\lambda)=-(\lambda^2+\varrho^2)\tilde{f}(\lambda).$$
 Under certain assumptions on $\alpha$ and $\beta$ the inversion and Plancherel formula for this transform take a nice form as described below.
 \begin{thm}[\cite{K2}]
 	\label{hpi}
 	Let $\alpha,\beta\in\R$, $\alpha>-1$ and $|\beta|\leq\alpha+1.$ Suppose $ c_{\alpha,\beta}(\lambda) $ denotes the Harish-Chandra $c $- function defined by 
 	$$ c_{\alpha,\beta}(\lambda)=\frac{2^{\varrho-i\lambda}\Gamma(\alpha+1)\Gamma(i\lambda)}{\Gamma\left(\frac{1}{2}(i\lambda+\varrho)\right)\Gamma\left(\frac{1}{2}(i\lambda+\alpha-\beta+1)\right)}$$ 
 	\begin{enumerate}
 		\label{planjj}
 		\item (Inversion) For $f\in C_0^{\infty}(\R)$ which is even we have 
 		$$f(r)=\frac{1}{2\pi}\int_{ 0}^{\infty}J_{\alpha,\beta}f(\lambda)\varphi_{\lambda}^{(\alpha,\beta)}(r)|c_{\alpha,\beta}(\lambda)|^{-2}d\lambda$$ 
 		\item (Plancherel) For $f,g\in C^{\infty}_0(\R)$ which are even, the following holds
 		$$\int_{ 0}^{\infty}f(r)\overline{ g(r)} {w}_{\alpha,\beta}(r)dr=\int_{ 0}^{\infty}J_{\alpha,\beta}f(\lambda)\overline{J_{\alpha,\beta}g(\lambda)}|c_{\alpha,\beta}(\lambda)|^{-2}d\lambda.$$  
 	\end{enumerate}
 	The mapping $f\rightarrowtail \tilde{f}$ extends  as an isometry from  $L^2(\R^+, {w}_{\alpha,\beta}(r)dr)$ onto
 	$L^2(\R^+,|c_{\alpha,\beta}(\lambda)|^{-2}d\lambda).$
 \end{thm}
  
  We will make use of this theorem in proving an analogue of Chernoff's theorem for the Laplace-Beltrami operator $\Delta_{X}$  in the next section.
 \section{Chernoff's theorem on noncompact symmetric  spaces of  rank one }
 In this section we prove our main theorem i.e., an analogue of Chernoff's theorem for $\Delta_{X}$. The main idea of the proof is to reduce the result for $\Delta_{X}$ to a result for Jacobi operator. So, first we indicate a proof of Chernoff's theorem for Jacobi operator. It has already been discussed in the work of Ganguly-Thangavelu \cite{GT2}.   
  \begin{thm}
  		\label{chernoffJ} 	Let $\alpha,\beta\in\R$, $\alpha>-1$ and $|\beta|\leq\alpha+1.$ Suppose $ f \in L^2(\R^+, {w}_{\alpha,\beta}(r)dr) $ is such that $ \mathcal{L}_{\alpha,\beta}^mf \in L^2(\R^+, {w}_{\alpha,\beta}(r)dr) $ for all $ m \in \Na $  and satisfies the Carleman condition 
  	$ \sum_{m=1}^\infty  \|  \mathcal{L}_{\alpha,\beta}^m f \|_2^{-1/(2m)} = \infty.$ If $\mathcal{L}_{\alpha,\beta}^mf(0)=0$ for all $m\geq 0$ then $f$ is identically zero.
  \end{thm}

  In \cite{GT2} the above result was proved under the assumption that $ f $ vanishes near $ 0 $ but a close examination of the proof reveals that the assumption is superfluous and the same is true as stated above.
In order to prove our main result, the following estimate for  the ratio of Harish-Chandra $c$-functions  is also needed. 
\begin{lem}
	\label{estC}
	Let $\alpha,\beta$ be as in \ref{ab} and $(p,q)$ be the pair of integers associated to $\delta\in \widehat{K_0}$. Then for any $\lambda\geq 0$ we have $$\frac{|c_{\alpha,\beta}(\lambda)|^2}{|c_{\alpha+p,\beta+q}(\lambda)|^2}|Q_{\delta}(i\lambda+\rho)|^{-2}\leq C$$ where $C$ is a constant independent of $ \lambda $ depending only on  the parameters $ (\alpha,\beta) $ and $ (p,q).$
\end{lem}
\begin{proof}
	First note that from the definition \ref{kos} of Kostant polynomials we have 
	$$|Q_{\delta}(i\lambda+\rho)|= \prod_{j=0}^{\frac{p+q}{2}}\left((B_1+j)^2+\frac14\lambda^2\right)^{\frac12}\prod_{j=0}^{\frac{p-q}{2}}\left((B_2+j)^2+\frac14\lambda^2\right)^{\frac12}$$ where $B_1=\frac12(\alpha+\beta+1)$ and $B_2=\frac12(\alpha-\beta+1)$.  From the above expression, it can be easily checked that $|Q_{\delta}(i\lambda+\rho)|/(2^{-1}\lambda)^p\rightarrow 1$ as $\lambda\rightarrow\infty$ so that
	\begin{equation}
	\label{c1}
	|Q_{\delta}(i\lambda+\rho)|\sim  2^{-p} \lambda^p,\ \ \ \ \lambda\rightarrow\infty. 
	\end{equation}
	Moreover, we also have
	$$|Q_{\delta}(i\lambda+\rho)|\geq \prod_{j=0}^{\frac{p+q}{2}}|B_1+j|\prod_{j=0}^{\frac{p-q}{2}}|B_2+j|=\text{constant}.$$
	Now using \cite[Lemma 2.4]{BP} we have 
	\begin{equation}
	\label{c2}
	\frac{|c_{\alpha,\beta}(\lambda)|^2}{|c_{\alpha+p,\beta+q}(\lambda)|^2} \sim \lambda^{2p},\ \ \ \lambda\rightarrow\infty
	\end{equation}
	which together with \ref{c1} implies that 
	$$\frac{|c_{\alpha,\beta}(\lambda)|^2}{|c_{\alpha+p,\beta+q}(\lambda)|^2}|Q_{\delta}(i\lambda+\rho)|^{-2}\sim 1,\ \ \ \lambda\rightarrow\infty.$$
	Also the  ratio in \ref{c2} being a continuous function of $\lambda$ is bounded near the origin. Hence the result follows.
	\end{proof}
 
\textit{Proof of Theorem \ref{C}:}  Let $f$ be as in the statement of the theorem \ref{C}. We complete the proof in the following steps.\\
\textit{Step 1:} Using Proposition \ref{H2} we write  
\begin{equation}
\widetilde{f}(\lambda, k)=\sum_{  \delta\in \widehat{K_0}}\sum_{j=1}^{d_{\delta}} F_{\delta,j}(\lambda)Y_{\delta,j}(k)
\end{equation}
where   $ F_{\delta,j}(\lambda) $ are the spherical harmonic coefficients of $ \tilde{f}(\lambda, \cdot) $ defined by
$$ F_{\delta,j}(\lambda)=\int_{K/M} \widetilde{f}(\lambda,k)Y_{\delta,j}(k)dk.$$ Fix $\delta\in \widehat{K_0}$ and $1\leq j\leq d_{\delta}.$ From the definition of the Helgason Fourier transform we have 
$$ F_{\delta,j}(\lambda)=\int_{K/M}\int_{G/K}f(x)e^{(-i\lambda+\rho)A(x,kM)}Y_{\delta,j}(kM)dxdk .$$
 Now using Fubini's theorem, in view of the Proposition \ref{fsp} the integral on the right hand side of above is equal to  
\begin{equation}
\label{p2}
 \int_{G/K}f(x)Y_{\delta,j}(kM)\Phi_{\lambda,\delta}(a_r)dx. 
\end{equation}
The function $ g_{\delta,j}(x) $ defined by
$$  g_{\delta,j}(x)=\int_{K}f(k'x)Y_{\delta,j}(k'M)dk',\ x\in X $$
is clearly $ K$-biinvariant, and hence by abuse of notation we write
$$ g_{\delta,j}(r)= \int_{K}f(k'a_r)Y_{\delta,j}(k'M)dk'.$$ 
  Now performing the integral in \ref{p2} using polar coordinates we obtain 
  \begin{equation}
  \label{p3}
  F_{\delta,j}(\lambda)=\int_{0}^{\infty} g_{\delta,j}(r)\Phi_{\lambda,\delta}(a_r) w_{\alpha,\beta}(r)dr 
  \end{equation}
Now recall that for each $\delta\in \widehat{K_0}$ there exist a pair of integers $(p, q)$ such that 
$$\Phi_{\lambda,\delta}(x)=Q_{\delta}(i\lambda+\rho)(\alpha+1)_p^{-1}(\sinh r)^p(\cosh r)^q \varphi_{\lambda}^{(\alpha+p,\beta+q)}(r).$$
By defining  \begin{equation}
f_{\delta,j}(r)=\frac{4^{-(p+q)}}{(\alpha+1)_p} g_{\delta,j}(r)(\sinh r)^{-p}(\cosh r)^{-q}
\end{equation}
and recalling  the definition of Jacobi transforms we obtain
\begin{equation}
\label{p4}
F_{\delta,j}(\lambda)=Q_{\delta}(i\lambda+\rho) J_{\alpha+p,\beta+q}(f_{\delta,j})(\lambda) 
\end{equation} 
\textit{Step 2:} In this step we estimate the $L^2$ norm of powers of Jacobi operator applied to $f_{\delta,j}$ in terms of the $L^2$ norm of corresponding powers of $\Delta_{X}$ applied to $f$.  Let $m\in \mathbb{N}.$  Note that the Plancherel formula \ref{planjj} for the Jacobi transform  yields 
 \begin{align*}
 	&\|\mathcal{L}^m_{\alpha+p,\beta+q}(f_{\delta,j})\|_{L^2(\R^+, {w}_{\alpha+p,\beta+q}(r)dr)}\\&=\left(\int_{ 0}^{\infty}(\lambda^2+\rho_\delta^2)^{2m}|J_{\alpha+p,\beta+q}(f_{\delta,j})(\lambda)|^2|c_{\alpha+p,\beta+q}(\lambda)|^{-2}d\lambda\right)^{\frac12}
 \end{align*} 
 where where $  \rho_\delta=\alpha+\beta+p+q+1.$ 
 In view of \ref{p4} the above integral reduces to 
 $$\left(\int_{ 0}^{\infty}(\lambda^2+\rho_\delta^2)^{2m} \, | F_{\delta,j}(\lambda)|^2 \,   |Q_{\delta}(i\lambda+\rho)|^{-2} \,c_{\alpha+p,\beta+q}(\lambda)|^{-2}d\lambda\right)^{\frac12}$$ which after recalling  the definition of $F_{\delta,j}(\lambda)$ reads  as 
 $$\left(\int_{ 0}^{\infty}(\lambda^2+\rho_\delta^2)^{2m}|Q_{\delta}(i\lambda+\rho)|^{-2}|\left|\int_K \widetilde{f}(\lambda,k)Y_{\delta,j}(k)dk\right|^2|c_{\alpha+p,\beta+q}(\lambda)|^{-2}d\lambda\right)^{\frac12}.$$  By an application of Minkowski's integral inequality, the above integral is dominated by 
 $$\int_{K}\left(\int_{ 0}^{\infty}(\lambda^2+\rho_\delta^2)^{2m}|Q_{\delta}(i\lambda+\rho)|^{-2} |\widetilde{f}(\lambda,k)|^2|c_{\alpha+p,\beta+q}(\lambda)|^{-2}d\lambda\right)^{\frac12}|Y_{\delta,j}(k)|dk.$$ Now using Cauchy-Schwarz inequality along with the fact that $\|Y_{\delta,j}\|_{L^2(K/M)}=1,$ we see that the above integral is bounded by 
 $$\left(\int_{K/M}\int_{ 0}^{\infty}(\lambda^2+\rho_\delta^2)^{2m}|Q_{\delta}(i\lambda+\rho)|^{-2} |\widetilde{f}(\lambda,k)|^2|c_{\alpha+p,\beta+q}(\lambda)|^{-2}d\lambda \  dk\right)^{\frac12}$$
 Since  $\frac{\lambda^2+\rho_\delta^2}{\lambda^2+\rho^2} = 1+ \frac{\rho_\delta^2-\rho^2}{\lambda^2+\rho^2} $ is a decreasing function  of $\lambda$  it follows that $\frac{\lambda^2+d^2}{\lambda^2+\rho^2}\leq C(\alpha,\beta)$ with $C(\alpha,\beta) = \frac{(\alpha+\beta+p+q+1)^2}{(\alpha+\beta+1)^2}.$ This together with the Lemma \ref{estC} yields the following estimate  for the integral under consideration:  for some constant $ C_1= C_1(\alpha,\beta)$
  $$   C_1^m \left(\int_{K/M}\int_{ 0}^{\infty}(\lambda^2+\rho^2)^{2m}  |\widetilde{f}(\lambda,k)|^2|c_{\alpha ,\beta }(\lambda)|^{-2}d\lambda \  dk\right)^{\frac12}.$$ 
  Finally, from the series of inequalities above, we obtain 
  \begin{equation}
   \|\mathcal{L}^m_{\alpha+p,\beta+q}(f_{\delta,j})\|_{L^2(\R^+, {w}_{\alpha+p,\beta+q}(r)dr)}\leq    C_1^m  \|\Delta_{X}^mf\|_2. 
  \end{equation}
  Hence from the hypothesis of the theorem it follows that 
  $$\sum_{m=1}^{\infty}\|\mathcal{L}^m_{\alpha+p,\beta+q}(f_{\delta,j})\|_{L^2(\R^+, {w}_{\alpha+p,\beta+q}(r)dr)}^{-\frac{1}{2m}}=\infty.$$ 
   \textit{Step 3:} Finally in this step we prove that $\mathcal{L}^m_{\alpha+p,\beta+q}(f_{\delta,j})(0)=0$ for all $m\geq 0.$ First recall that 
  $$f_{\delta,j}(r)=\frac{4^{-(p+q)}}{(\alpha+1)_p}(\sinh r)^{-p}(\cosh r)^{-q}\int_{K}f(ka_r)Y_{\delta,j}(kM)dk.$$  As  $\sinh r$ has a zero at the origin and $\cosh 0=1$,  if we can show that as a function of $r$, the integral $\int_{K}f(ka_r)Y_{\delta,j}(kM)dk$ has a zero of infinite order at the $ 0 ,$ then we are done. Now note that for any $m\in \mathbb{N}$ 
  $$\frac{d^m}{dr^m}\int_{K}f(ka_r)Y_{\delta,j}(kM)dk=\int_{K}\frac{d^m}{dr^m}f(ka_r)Y_{\delta,j}(kM)dk.$$ But by definition of the vector fields on $G$, writing $a_r=\exp (rH)$ we have 
  $$\frac{d^m}{dr^m}f(ka_r)|_{r=0}=\frac{d^m}{dr^m}f(k.\exp(rH))|_{r=0}= H^mf(k).$$  Hence by the hypothesis on $ f $ we obtain $\frac{d^m}{dr^m}f(ka_r)|_{r=0}=0$ for all $m$. Finally, proving  $\mathcal{L}^m_{\alpha+p,\beta+q}(f_{\delta,j})(0)=0$ is a routine matter:  repeated application of L'Hospital rule gives the desired result. 
  
  Therefore,   $f_{\delta,j}$ satisfies all  the hypothesis of the Proposition \ref{chernoffJ} which allows us to  conclude that $f_{\delta,j}=0$ i.e., $F_{\delta,j}=0.$  As this  is true for every $\delta\in \widehat{K_0}$ and $1\leq j\leq d_{\delta}$ we get  $f=0$ completing the proof of Theorem 1.2.

  \section{Compact symmetric spaces}
  Our aim in this section is to prove an analogue of Chernoff's theorem on  compact symmetric spaces of rank one. To begin with, we first recall briefly some   necessary  background material  on rank one compact symmetric spaces. 
 Let $S$ be a compact Riemannian manifold equipped with a Riemannian metric $d_{S}$. We say that $S$ is a two point homogeneous space if for any $x_j,y_j\in S,\ j=1,2$ with $d_{S}(x_1,x_2)=d_{S}(y_1,y_2)$, there exists $g\in I(S)$, the group of isometries of $S$ such that $g.x_1=y_1$, and $g.x_2=y_2$ where $g.x$ denotes the usual action of $I(S)$ on $S$. It is well known that compact rank one symmetric spaces are compact two point homogeneous spaces (see Helgason\cite{H3}). Also these two point homogeneous spaces are completely classified by H-.C. Wang \cite{W}. So, following Wang any compact rank one symmetric space $S$ is one of the following:  
  \begin{enumerate}
     \item the sphere $\mathbb{S}^q \subset \R^{q+1},~q\geq 1$;
      \item the real projective space $P_q(\R),~q\geq 2;$
    \item  the complex projective space $P_l(\mathbb{C}), ~l\geq 2$;
      \item the quaternionic projective space $P_l(\mathbb{H}),~l\geq 2;$
      \item the Cauchy projective plane $P_2(\mathbb{C} a y).$
  \end{enumerate}
 We describe the necessary preliminaries and prove Theorem 1.9 in each of the above five cases separately.  We start with a brief description of Jacobi polynomial expansions in the following subsection.
 \subsection{Jacobi polynomial expansion:}
 Let $\alpha,\beta>-1.$  The  Jacobi polynomials $P_n^{\alpha,\beta}$ of degree $n \geq 0$ and type $(\alpha, \beta)$ are defined by
 \begin{align} \label{jacobi}
 (1-x)^{\alpha} \, (1+x)^{\beta}P_n^{\alpha,\beta}(x)=\frac{(-1)^n}{2^n n!} \frac{d^n}{dx^n}\{(1-x)^{n+\alpha} \, (1+x)^{n+\beta}\} ,~x\in (-1,1).
 \end{align} 
 By making a change of variable $x=\cos\theta$, it is convenient to work with the Jacobi trigonometric polynomials
 \begin{align}
 \mathcal{P}_n^{(\alpha,\beta)}(\theta) =C(\alpha,\beta,n) P_n^{(\alpha,\beta)}(\cos \theta),
 \end{align}
 where $C(\alpha,\beta,n)$ is the normalising constant, explicitly given by  \begin{align} \label{db}
 C(\alpha,\beta,n)^2=\frac{(2n+\alpha+\beta+1)\Gamma(n+1)\Gamma(n+\alpha+\beta+1) }{\Gamma(n+\alpha+1)\Gamma(n+\beta+1)}.
 \end{align} Also it is worth pointing out that these polynomials are closely related to Gegenbauer's polynomials by the following formula 
 \begin{equation}
 \label{gj}
 C_k^{\lambda}(t)=\frac{\Gamma(\lambda+\frac12) \Gamma(k+2\lambda)}{\Gamma(2\lambda)\Gamma(k+\lambda+\frac12)}P_k^{(\lambda-\frac12,\lambda-\frac12)}(t), \lambda >-\frac12, t\in(-1,1).
 \end{equation}
These Jacobi trigonometric polynomials are the eigenfunctions of the Jacobi differential operator given by 
\[\mathbb{L}_{\alpha,\beta}=-\frac{d^2}{d \theta^2}-\frac{\alpha-\beta+(\alpha+\beta+1) \cos \theta}{ \sin \theta}+\left(\frac{\alpha+\beta+1}{2}\right)^2\]
 with eigenvalues $(n+\frac{\alpha+\beta+1}{2})^2$ i.e., $$\mathbb{L}_{\alpha,\beta} \mathcal{P}_n^{(\alpha,\beta)}=\left(n+\frac{\alpha+\beta+1}{2}\right)^2\,  \mathcal{P}_n^{(\alpha,\beta)},$$and  $\{\mathcal{P}_n^{(\alpha,\beta)}:n\geq0\}$ forms an orthonormal basis for the weighted $L^2$ space $L^2(\tilde{w}_{\alpha,\beta}):=L^2((0,\pi), \tilde{w}_{\alpha,\beta}(\theta)d\theta)$ where the weight is given by 
$$\tilde{w}_{\alpha,\beta}(\theta)=\left(\sin \frac{\theta}{2}\right)^{2\alpha+1} \, \left(\cos \frac{\theta}{2}\right)^{2\beta+1}.$$ As a consequence we have the following Plancherel formula valid for $f\in L^2(\tilde{w}_{\alpha,\beta})$
\begin{equation}
\label{jplan}
\int_{ 0}^{\pi}|f(\theta)|^2  \tilde{w}_{\alpha,\beta}(\theta)d\theta=\sum_{n=0}^{\infty}|\mathcal{J}_{\alpha,\beta}f(n)|^2
\end{equation}
where $\mathcal{J}_{\alpha,\beta}f(n)$ denotes the Fourier-Jacobi coefficients defined by 
$$\mathcal{J}_{\alpha,\beta}f(n)=\int_0^{\pi} f(\theta) \, \mathcal{P}_n^{(\alpha,\beta)}(\theta) \, \tilde{w}_{\alpha,\beta}(\theta)d\theta,\ \ \ n\geq 0.$$ 
We have the following version of Chernoff's theorem using the iterates of the Jacobi operator proved in Ganguly-Thangavelu \cite{GT1}.
\begin{thm}
		\label{chernoffJp} 	Let  $\alpha,\beta>-1$. Suppose $ f \in L^2( \tilde{w}_{\alpha,\beta} ) $ is such that $ \mathbb{L}_{\alpha,\beta}^mf \in L^2( \tilde{w}_{\alpha,\beta} ) $ for all $ m \in \Na $  and satisfies the Carleman condition 
	$ \sum_{m=1}^\infty  \|  \mathbb{L}_{\alpha,\beta}^m f \|_2^{-1/(2m)} = \infty.$ If $\mathbb{L}_{\alpha,\beta}^mf(0)=0$ for all $m\geq 0$ then $f$ is identically zero.
\end{thm} 

This is the analogue of Theorem 3.1 for Jacobi polynomial expansions which plays an important role in proving Theorem 1.9 for compact Riemannian symmetric spaces.

\subsection{The unit sphere $\mathbb{S}^q$} Let $q\geq 2$. The unit sphere in $\mathbb{R}^{q+1}$ is given by $$\mathbb{S}^{q}:=\{\xi \in \R^{q+1}:\xi_1^2+\cdots+\xi_{q+1}^2=1\}.$$ The spherical harmonic decomposition reads as 
$$L^2(\mathbb{S}^q)=\displaystyle\bigoplus_{n=0}^{\infty}\mathcal{H}_{n}(\mathbb{S}^q)$$ where $\mathcal{H}_{n}(\mathbb{S}^q)$ denotes the set of spherical harmonics of degree $n$. Now, for our purposes it is more convenient to work with the geodesic polar coordinate system on $\mathbb{S}^q$. Note that given $\xi\in \mathbb{S}^q$, we can write $\xi= (\cos\theta) e_1+\xi_1^{'}(\sin\theta) e_2+...+\xi_{q}^{'}(\sin\theta) e_{q+1}$ for some $\theta \in (0,\pi)$ and $\xi^{'}=(\xi_1^{'},...,\xi_{q}^{'})\in \mathbb{S}^{q-1}$ where $\{e_1,e_2,...,e_{q+1}\}$ is the standard basis for $\mathbb{R}^{q+1}.$ This observation drives us to consider the map $\varphi:(0,\pi)\times \mathbb{S}^{q-1}\rightarrow \mathbb{S}^{q}$ defined by 
$$\varphi(\theta, \xi')=(\cos \theta, \xi_1' \sin \theta, \dots, \xi_q' \sin \theta)$$ which induces the geodesic polar coordinate system on $\mathbb{S}^q$. This also provides a polar decomposition of the normalised measure $d\sigma_q$ on $\mathbb{S}^{q}$ as follows:  Given a suitable function $f$ on $\mathbb{S}^q$ we have 
 \[\int_{\mathbb{S}^q}f(\xi) d\sigma_q(\xi)=\int_{0}^{\pi} \int_{\mathbb{S}^{q-1}}F(\theta,\xi') \, (\sin \theta)^{q-1} d\sigma_{q-1}(\xi^{'}) d\theta\] where $F=f\circ \varphi.$ Also in this coordinate system, we have the following representation of the Laplace-Beltrami operator 
 $$\Delta_{\mathbb{S}^q}=-\frac{\partial^2}{\partial\theta^2}-(q-1)\cot\theta \frac{\partial}{\partial\theta}+\frac14(q-1)^2-\sin^{-2}\theta\tilde{\Delta}_{\mathbb{S}^{q-1}}$$
   The following theorem gives a representation of the spherical harmonics in this polar coordinate system.
 \begin{thm}\cite[Theorem 2.4]{K}
 	\label{polarep}
 	For $n\geq0$ we have the following orthogonal decomposition
 	$$\mathcal{H}_n(\mathbb{S}^q)=\displaystyle\bigoplus_{l=0}^n\mathcal{H}_{n,l}(\mathbb{S}^q)$$ where the subspaces $\mathcal{H}_{n,l}(\mathbb{S}^q)$ are irreducible and invariant under $SO(q)$. Moreover, functions in $\mathcal{H}_{n,l}(\mathbb{S}^q)$ can be represented as 
 $$S(\xi)=(\sin\theta)^lC^{q/2-1/2+l}_{n-l}(\cos\theta)S^{'}_l(\xi^{'})$$ where $\xi=\varphi(\theta,\xi^{'})$ and $S_{l}^{'}\in \mathcal{H}_l(\mathbb{S}^{q-1}).$
 \end{thm}
In view of the above theorem we have the orthogonal decomposition $$L^2(\mathbb{S}^q)=\bigoplus_{n=0}^{\infty}\bigoplus_{l=0}^n\mathcal{H}_{n,l}(\mathbb{S}^q).$$ Now we set $$S_{n,l,k}(\xi)=a_{n,l} \, (\sin \theta)^l \, C_{n-l}^{l+\frac{q-1}{2}}(\cos \theta) \, S^{'}_{k,l}(\xi')$$ where $\{S^{'}_{l,k}: 1\leq k\leq N(l)\}$ is an orthonormal basis for $\mathcal{H}_{l}(\mathbb{S}^{q-1}).$ Here $a_{n,l}$ is the normalising constant so that $\|S_{n,l,k}\|_{L^2(\mathbb{S}^q)}=1$ and it is explicitly given by 
\begin{align}
\label{anl}
	a_{n,l}=\frac{2^{-(l+\frac{q-1}{2})}\Gamma(2l+q-1)\Gamma(n+\frac q2) }{\Gamma(l+\frac q2)\Gamma(n+l+q-1)} \,  C\left(l+\frac{q-2}{2},l+\frac{q-2}{2},n-l\right).
\end{align}
\begin{thm}
	\label{chsphere}
	Let $f \in C^{\infty}(\mathbb{S}^q)$ be such that $\Delta^m_{\mathbb{S}^d}f\in L^2(\mathbb{S}^q)$ for all $m\geq0 $ and satisfies
	\[\sum_{m=1}^{\infty}\|\Delta^m_{\mathbb{S}^q}f\|_2^{-\frac{1}{2m}} =\infty.\]
	If $\frac{\partial^m}{\partial\theta^m} \big|_{\theta=0} F(\theta, \xi')=0$ for all $m\geq 0$  and for all $\xi' \in \mathbb{S}^{q-1},$ then $f$ is identically zero.
	
\end{thm}
\begin{proof}
	Let $f$ be as in the statement of the theorem. For $n\geq0$, let $P_nf$ denote the projection of $f$ onto the space $\mathcal{H}_n(\mathbb{S}^q).$ Then from the above observations we have 
	\begin{equation}
	\label{sproj}
	P_nf= \sum_{l=0}^n\sum_{k=1}^{N(l)}(f, S_{n,l,k})_{L^2}S_{n,l,k}.
	\end{equation}
 Also since $f\in L^2(\mathbb{S}^q)$ we have 
 $$f=\sum_{n=0}^{\infty} P_nf=\sum_{n=0}^{\infty}\sum_{l=0}^n\sum_{k=1}^{N(l)} ( f, S_{n,l,k})_{L^2(\mathbb{S}^q)} S_{n,l.k}$$ 
 By interchanging the  summations, we observe that 
 \begin{align*}
 &f=\sum_{l=0}^{\infty}\sum_{n=l}^{\infty}\sum_{k=1}^{N(l)} ( f, S_{n,l,k})_{L^2(\mathbb{S}^q)} S_{n,l.k}\\
 &=\sum_{l=0}^{\infty}\sum_{n=0}^{\infty}\sum_{k=1}^{N(l)} (f, S_{n+l,l,k})_{L^2(\mathbb{S}^q)} S_{n+l,l.k}.
 \end{align*} 
	In view of this, to prove the theorem it is enough to prove that $(f, S_{n+l,l,k})_{L^2(\mathbb{S}^q)} = 0 $ for all $n,l,k.$ To start with, let us first fix $n,l$ and $k$.  From the expansion \ref{sproj} we observe that 
	\begin{equation}
	\label{a11}
	  (P_nf, S_{n+l,l,k})_{L^2(\mathbb{S}^q)}=( f, S_{n+l,l,k})_{L^2(\mathbb{S}^q)}.
	\end{equation}
	Next we use the expression for $S_{n+l,l,k}$ to show that these coefficients are nothing but Jacobi coefficients of a suitable function. 
	In order to do so, we write the integral on $\mathbb{S}^q$ in polar coordinates to obtain 
	$$( f, S_{n+l,l,k})_{L^2(\mathbb{S}^q)}=\int_0^{\pi} \int_{\mathbb{S}^{q-1}} F(\theta,\xi') \, a_{n+l,l} \,  (\sin \theta)^{l+q-1} \, C_{n}^{l+\frac{q-1}{2}}(\cos \theta) \, S^{'}_{k,l}(\xi')  \, d\sigma_{q-1}({\xi'}) \, d\theta$$
	where $F:=f\circ \varphi.$  Now using \ref{db}, \ref{gj} and \ref{anl}, a simple calculation yields 
	\begin{align}
	a_{n+l,l}C_{n}^{l+\frac{q-1}{2}}(\cos \theta)=2^{-(l+\frac{q-1}{2})} C(l+\frac q2-1,l+\frac q2-1,n)
 P_n^{(l+\frac q2-1,l+\frac q2-1)}(\cos \theta)
	\end{align}
	which transforms the above equation into 
	\begin{equation}
	\label{a111}
	( f, S_{n+l,l,k})_{L^2(\mathbb{S}^q)}=2^{-(l+\frac{q-1}{2})}\int_{0}^{\pi}F_{k,l}(\theta)  \,(\sin \theta)^{l+q-1} \,\mathcal{P}_n^{(l+\frac d2-1,l+\frac d2-1)}(\theta) \, d\theta
	\end{equation} 
	where we have defined 
	$$F_{k,l}(\theta):=\int_{\mathbb{S}^{q-1}}F(\theta,\xi^{'})S^{'}_{k,l}(\xi')  \, d\sigma_{q-1}({\xi'}).$$
	Now letting $\alpha=l+\frac d2-1$ and writing $\sin\theta=2\sin\frac{\theta}{2}\cos\frac{\theta}{2}$ we see that
	$$(\sin\theta)^{l+q-1}=2^{l+q-1} (\sin\theta)^{-l}w_{\alpha,\alpha}(\theta)$$
	which together with \ref{a111} yields
	\begin{equation}
	( f, S_{n+l,l,k})_{L^2(\mathbb{S}^q)}= \mathcal{J}_{\alpha,\alpha}(g_{k,l})(n)
	\end{equation} 
where $g_{k,l}(\theta):=  2^{\frac{q-1}{2}}(\sin\theta)^{-l}F_{k,l}(\theta).$  
 
In view of  the Plancherel formula \ref{jplan} and the relation \ref{a11} we have 
\begin{align}
\label{a1111}
\|\mathbb{L}_{\alpha,\alpha}^mg_{l,k}\|_2^2&= \sum_{n=0}^{\infty}\left(n+\frac{2\alpha+1}{2}\right)^{4m} \, |\mathcal{J}_{\alpha,\alpha}(g_{l,k})(n)|^2\nonumber\\
&= \sum_{n=0}^{\infty}\left(n+\frac{2l+q-1}{2}\right)^{4m} \, \left| \int_{\mathbb{S}^d} P_nf(\xi) S_{n+l,l,k}(x) d \sigma_{q}(\xi)\right|^2 ,
\end{align}
By Cauchy-Schwarz inequality we note that 
$$\left| \int_{\mathbb{S}^d} P_nf(\xi) S_{n+l,l,k}(\xi) d \sigma_{q}(\xi)\right|^2\leq \|P_nf\|_{L^2(\mathbb{S}^q)}^2.$$ 
Finally, using the fact that $n+\frac12(2\alpha+1)= n+\frac12(2l+q-1)\leq \left(n+\frac{q-1}{2}\right)\left(1+\frac{2l}{q-1}\right),$ from \ref{a1111} we get the estimate
 $$\|\mathbb{L}_{\alpha,\alpha}^mg_{l,k}\|_2^2\leq \left(1+\frac{2l}{q-1}\right)^{4m} \sum_{n=0}^{\infty}\left(n+\frac{q-1}{2}\right)^{4m} \, \|P_nf\|^2_{L^2(\mathbb{S}^q)}.$$
 Therefore, we have proved
 \[\|\mathbb{L}_{\alpha,\alpha}^mg_{l,k}\|_2 \leq  \left(1+\frac{2l}{q-1}\right)^{2m} \|\Delta^m_{\mathbb{S}^q}f\|_{L^2(\mathbb{S}^q)}\]
 which by  the hypothesis on the function $ f ,$ implies that
 \begin{equation}
 \label{carcon}
 \sum_{m=1}^{\infty}\|\mathbb{L}_{\alpha,\alpha}^mg_{l,k}\|_2^{-\frac{1}{2m}} =\infty.
 \end{equation} 
 Since $ g_{l,k}(\theta) $ is related to $ F(\theta,\xi') $ via the integral
 $$g_{l,k}(\theta)= 2^{\frac{q-1}{2}}(\sin\theta)^{-l}\int_{\mathbb{S}^{q-1}}F(\theta,\xi^{'})S^{'}_{k,l}(\xi')  \, d\sigma_{q-1}({\xi'})$$
  the hypothesis  $\frac{\partial^m}{\partial\theta^m} \big|_{\theta=0}F(\theta, \xi')=0$ for all $m\geq 0$ allows us to conclude  that $\mathbb{L}_{\alpha,\alpha}^mg_{l,k}(0)=0$ for all $m\geq 0.$ Hence $g_{l,k}$ satisfies the hypotheses of  Theorem \ref{chernoffJp} and hence we conclude that  $g_{l,k} = 0 $ and consequently $ ( f, S_{n+l,l,k})_{L^2(\mathbb{S}^q)}=0.$  As this is true for any  $n,l,k$, we conclude that $f=0$ 
  completing the proof of the theorem.
 \end{proof}
 
 \subsection{The real projective spaces $P_{q}(\mathbb{R})$}  
 Let $O(q)$ denote the group of $q\times q$ orthogonal matrices. Then $P_q(\mathbb{R})$ can be identified with $SO(q+1)/ O(q)$ which makes this a compact symmetric space. Now  it is well-known that the real projective space $P_q(\mathbb{R})$ can be obtained from $\mathbb{S}^q$ by identifying the antipodal points i.e., $P_q(\mathbb{R})=\mathbb{S}^q/\{\pm I\}$ and the projection map $s\rightarrow \pm s$ from $\mathbb{S}^q$ to $P_q(\mathbb{R})$ is locally an isometry. So, the  functions on $P_q(\mathbb{R})$ can be viewed as  even functions on the corresponding sphere $\mathbb{S}^q$ and if $f_e$ is the even function on $\mathbb{S}^q$ corresponding to the function $f$ on $P_q(\mathbb{R})$ then $\Delta_{P_q(\mathbb{R})}f=\Delta_{\mathbb{S}^q}f_e.$  Hence the analogue of Chernoff's theorem on $P_q(\mathbb{R})$ follows directly from the case of sphere.
 
 \subsection{The other projective spaces  $P_l(\mathbb{C})$, $P_l(\mathbb{H}),$ and $P_2(\mathbb{C} a y)$}
 
As pointed out by T. O. Sherman in \cite{S}, analysis on these three projective spaces is quite similar. Closely following the notations of \cite{S} (see also \cite{CRS}), we first describe the appropriate polar coordinate representation of these spaces and then as in the sphere case we prove the Chernoff's theorem for the associated Laplace-Beltrami operators. To begin with, let $S$ denote any of these three spaces $P_l(\mathbb{C})$, $P_l(\mathbb{H}),$ and $P_2(\mathbb{C} a y)$. Suppose $\tilde{\Delta}_S$ denotes the corresponding Laplace-Beltrami operator. Let $d\mu$  denote the normalised Riemann measure on $S$. We have the following orthogonal decomposition:
\[L^2(S,d\mu):=L^2(S)=\bigoplus_{n=0}^{\infty} \mathcal{H}_{n}(S),\] where $\mathcal{H}_{n}(S)$ are finite dimensional and eigenspaces of $\tilde{\Delta}_S$ with eigenvalue $-n(n+k+q)$ where $q=2,4,8,~k=l-2,2l-3,3,$ for $P_l(\mathbb{C}),~P_l(\mathbb{H})$ and $P_2(\mathbb{C} a y),$ respectively. However, it is convenient to  work with $\Delta_S:=-\tilde{\Delta}_{S}+\rho_S^2$ where $\rho_S:=\frac12(k+q).$  As a result $\mathcal{H}_n(S)$ becomes eigenspaces of $\Delta_{S}$ with eigenvalue $\left(n+\frac{k+q}{2}\right)^2.$  

Let $\Omega:=\{x\in\mathbb{R}^{q+1}:|x|\leq1\}$ be the closed unit ball in $\mathbb{R}^{q+1}.$ We consider a weight function  $w$ defined by $w(r):=r^{-1}(1-r)^k$ for $0<r\leq 1.$ With these notations we have the following result proved in \cite[Lemma 4.15]{S}.

\begin{prop}
	\label{e}
	There is a bounded linear map $E:L^1(S) \to L^1(\Omega, w(|x|) dx)$ satisfying
	\begin{enumerate}
		\item For $f\in L^1(S)$,
		\[\int_S f d\mu=\int_{\Omega}E(f)(x) \, w(|x|) \, dx\]
		
		\item The norm of $E$ as a map from $L^p(S)$ to $L^p(\Omega, w(|x|) dx)$ is $1~(1\leq p\leq \infty).$	
	\end{enumerate}
\end{prop}
 
 The integration formula in the above proposition is very useful. In fact, integrating the right hand side of that formula in polar coordinates we have 
 $$\int_{\Omega}E(f)(x) \, w(|x|) \, dx=\int_{0}^{1}\int_{\mathbb{S}^q}E(f)(r\xi)w(r)r^qd\sigma_{q}(\xi)dr.$$ Now a change of variables  $r= \sin^2 (\theta/2)$ allows us to conclude that
 \begin{align}
 	\int_S f d\mu= \, \int_0^{\pi} \int_{S^q} F(\theta, \xi) \left(\sin \frac{\theta}{2}\right)^{2q-1} \, \left(\cos \frac{\theta}{2}\right)^{2k+1} \,  d\theta \, d\sigma_{q}(\xi),
 \end{align}
 where $F(\theta, \xi)=E(f)(\sin^2 (\theta /2)\, \xi)$.  
 In \cite{S} Sherman has described the image of $\mathcal{H}_n(S)$ under the map $E.$  It has been proved that  $E(\mathcal{H}_n(S))=\mathcal{H}_{n}(\Omega,w)$  where  $\mathcal{H}_{n}(\Omega,w)$ is the orthocomplement of $\mathbb{P}_{n-1}(\Omega)$ in $\mathbb{P}_n(\Omega)$ with respect to the inner product in $L^2(\Omega, w(|x|) dx).$ Here $\mathbb{P}_n(\Omega)$ denotes the set of all polynomials on $\Omega$ of degree up to $n$. Also note that in these trigonometric polar coordinates we can identify $\Omega$ with $\Omega_0:= (0,\pi)\times\mathbb{S}^q$ and  $$d\omega(\theta,\xi):= \left(\sin \frac{\theta}{2}\right)^{2q-1} \, \left(\cos \frac{\theta}{2}\right)^{2k+1} \,  d\theta \, d\sigma_{q}(\xi)$$ is the corresponding measure on $\Omega_0.$ Basically in view of this trigonometric polar coordinates we have  $\mathcal{H}_{n}(\Omega,w)=\mathcal{H}_n(\Omega_0, \omega).$ These spaces are eigenspaces of the following differential operator
  $$\Lambda_{S}=-\frac{\partial^2}{\partial\theta^2}-\frac{(q-1-k)+(q+k)\cos\theta}{\sin\theta}\frac{\partial}{\partial\theta}- \frac{1}{\sin^2 (\theta/2)}\tilde{\Delta}_{\mathbb{S}^q} + \left(\frac{k+q}{2}\right)^2$$ with eigenvalues $(n+\frac{q+k}{2})^2.$ The relation between this operator and the Laplace-Beltrami operator is described in the following proposition.
  \begin{prop}
  	\label{e1}
  	Let $f\in C^{2}(S)$ and $E$ be as in the Proposition\ref{e}. Then we have $$E(\Delta_{S}f)= \Lambda_{S} E(f).$$   
  \end{prop}
 For a proof of this fact we refer the reader to  \cite[Lemma 4.25]{S}.
 Thus we have the following orthogonal decomposition  
 $$L^2(\Omega_0, d\omega)=\bigoplus_{n=0}^{\infty}\mathcal{H}_{n}(\Omega_0,\omega).$$ Moreover, $\mathcal{H}_{n}(\Omega_0,\omega)$ admits a further decomposition as $\mathcal{H}_{n}(\Omega_0,\omega)=\bigoplus_{j=0}^{n}\mathcal{H}_{n,j}(\Omega_0,\omega)$  where each $\mathcal{H}_{n,j}(\Omega_0,\omega)$ is irreducible under $SO(q+1)$ and spanned by $\{Q_{n,j,l}: 1\leq l\leq N(j)\}$ (see \cite[Theorem 4.22]{S})  where  for $x=\sin^2(\theta/2) \,\xi,~\theta\in (0,\pi)~\text{and}~\xi\in \mathbb{S}^q$ 
 \begin{align*}
  Q_{n,j,l}(x)&= b_{n,j}\left(\sin\frac{\theta}{2}\right)^{2j}P_{n-j}^{(k,q-1+2j)}\left(2\sin^2(\theta/2)-1\right)S_{j,l}(\xi)\\
  &=(-1)^{n-j}b_{n,j}\left(\sin\frac{\theta}{2}\right)^{2j}P_{n-j}^{(q-1+2j,k)}(\cos\theta)S_{j,l}(\xi). 
 \end{align*}
 In the second equality we have used the symmetry relation for Jacobi polynomials i.e., $P^{(\alpha,\beta)}_{n}(-x)=(-1)^nP^{(\beta,\alpha)}_{n}(x).$
  Here $\{S_{j,l}:1\leq l\leq N(j)\}$ a basis for $\mathcal{H}_{j}(\mathbb{S}^q)$, the spherical harmonics of degree $l$ on $\mathbb{S}^q.$ The constants $b_{n,j}$ appearing in the above expression are chosen so that $\|Q_{n,j,l}\|_2=1$. In fact, it can be checked that 
$b_{n,j}= C(q-1+2j,k, n-j). $ So, clearly $\{Q_{n,j,l}:n,l\geq0, 1\leq N(l)\}$ forms an orthonormal basis for $L^2(\Omega_0, d\omega).$ Now we are ready to state and prove an analogue of Chernoff's theorem on $S$.
 \begin{thm}
		\label{chprojs}
	Let $f \in C^{\infty}(S)$ be such that $\Delta^m_{S}f\in L^2(S)$ for all $m\geq0.$ Assume that 
	\[\sum_{m=1}^{\infty}\|\Delta^m_{S}f\|_2^{-\frac{1}{2m}} =\infty.\]
	If the function $ F $ defined by  $F(\theta, \xi) =E(f)( \sin^2(\theta/2) \xi) $ satisfies $\frac{\partial^m}{\partial\theta^m} \big|_{\theta=0} F(\theta, \xi)=0$ for all $m\geq 0$  and for all $\xi \in \mathbb{S}^{q},$ then $f$ is identically zero.
\end{thm}
 
  \begin{proof}
  	Given a function $f$ with the property as in the statement of the theorem, we write $E(f)(\sin^2 (\theta/2)\, \xi)=F(\theta,\xi),~(\theta,\xi)\in \Omega_0.$ So, the analysis, described above allow us to write the projection of $F$ onto $\mathcal{H}_{n}(\Omega_0, \omega)$  as 
  	$$P^{S}_nF= \sum_{j=0}^{n}\sum_{l=1}^{N(j)}(F, Q_{n,j,l})Q_{n,j,l}.$$  
  	Now as in the sphere case, it is not hard to check that 
  	\begin{equation}
  	F=\sum_{j=0}^{\infty}\sum_{n=0}^{\infty}\sum_{l=1}^{N(j)} (F, Q_{n+j,j,l})_{L^2(\Omega_0,d\omega)} Q_{n+j,j,l}.
  	\end{equation}
  	Clearly, for each $n\geq 0$ we have 
  	\begin{equation}
  	\label{b1}
  	 (P^S_nF, Q_{n+j,j,l})_{L^2(\Omega_0,d\omega)}= (F, Q_{n+j,j,l})_{L^2(\Omega_0,d\omega)}.
  	\end{equation} 
  	As in the case of sphere, we will show that the right hand side of the above equation can be expressed as Jacobi coefficient of a suitable function related to $F.$ By definition, we have 
  	\begin{equation}
  	(F, Q_{n+j,j,l})=\, \int_0^{\pi} \int_{S^q} F(\theta, \xi)Q_{n+j,j,l}((\sin^2\frac{\theta}{2})\xi ) \left(\sin \frac{\theta}{2}\right)^{2q-1} \, \left(\cos \frac{\theta}{2}\right)^{2k+1} \,  d\theta \, d\sigma_{q}(\xi).
  	\end{equation} 
  	Now using the expression for $Q_{n+j,j,l}$ we have 
  	\begin{equation}
  	(F, Q_{n+j,j,l})= (-1)^{j}b_{n+j,j}\int_0^{\pi} F_{j,l}(\theta)\left(\sin\frac{\theta}{2}\right)^{2j}P_n^{(q-1+2j,k)}(\cos\theta)\left(\sin \frac{\theta}{2}\right)^{2q-1} \, \left(\cos \frac{\theta}{2}\right)^{2k+1} \,  d\theta.
  	\end{equation} where $ F_{j,l} $ are defined by 
  	$$F_{j,l}(\theta):= \int_{S^q}F(\theta,\xi)S_{j,l}(\xi)d\sigma_{q}(\xi).$$
  	Writing $g_{j,l}(\theta)=(-1)^jF_{j,l}(\theta)(\sin\frac{\theta}{2})^{-2j}$ and using the definition of Jacobi coefficients we have 
  	$$(F, Q_{n+j,j,l})_{L^2(\Omega_0,d\omega)}=\mathcal{J}_{\alpha,\beta}(g_{j,l})(n)$$ where 
  	$\alpha=q-2j+k$ and $\beta=k$. Now using the Plancherel formula \ref{jplan} along with \ref{b1} we obtain
  	\begin{align}
  		\|\mathbb{L}_{\alpha,\beta}^mg_{j,l}\|_2^2&= \sum_{n=0}^{\infty}\left(n+\frac{\alpha+\beta+1}{2}\right)^{4m} \, |\mathcal{J}_{\alpha,\beta}(g_{j,l})(n)|^2\nonumber\\
  		&= \sum_{n=0}^{\infty}\left(n+\frac{q-2j+2k+1}{2}\right)^{4m} \, |(P^S_nF, Q_{n+j,j,l})|^2.		
  	\end{align}
  	But $|(P^S_nf, Q_{n+j,j,l})|\leq \|P^S_nf\|_{L^2(\Omega_0,d\omega)}$ and $\left(n+\frac{q-2j+2k+1}{2}\right)\leq C(n+\frac{q+k}{2})$  so that we have
  	\begin{align}
  		\|\mathbb{L}_{\alpha,\beta}^mg_{j,l}\|_2^2\leq C^{4m}\sum_{n=0}^{\infty}\left(n+\frac{q+k}{2}\right)^{4m}\|P^S_nf\|_2^2 =C^{4m} \|\Lambda_{S}^mE(f)\|_2^2
  	\end{align}
 In view of the Proposition \ref{e1} we have $E(\Delta_{S}^mf)= \Lambda_{S}^m E(f)$ and using the fact that operator norm of $E$ is one (see Proposition \ref{e}) we have 
  	$$\|\mathbb{L}_{\alpha,\beta}^mg_{j,l}\|_2^2\leq C^{2m}\|\Delta^m_{S}f\|_2$$
  Hence the given condition $\sum_{m=1}^{\infty}\|\Delta^m_{S}f\|_2^{-\frac{1}{2m}} =\infty$ allows us to conclude that 
  \begin{equation}
  	\sum_{m=1}^{\infty}\|\mathbb{L}_{\alpha,\beta}^mg_{j,l}\|_2^{-\frac{1}{2m}} =\infty.
  \end{equation}
  Also using the hypothesis $\frac{\partial^m}{\partial \theta^m} \big|_{\theta=0} F(\theta, \xi)=0$ for all $m\geq 0, \xi \in \mathbb{S}^q,$  a simple calculation shows that $\mathbb{L}_{\alpha,\beta}^mg_{j,l}(0)=0$ for all $m\geq0.$ Hence by Theorem \ref{chernoffJp}, we have $g_{j,l}=0$ whence $(F, Q_{n+j,j,l})_{L^2(\Omega_0,d\omega)}=0.$  As this  is true for all $n,j,l$ we conclude  $f=0$. This completes the proof of the theorem.
  	\end{proof}
  
\section*{Acknowledgments}The first author is supported by Int. Ph.D. scholarship from Indian Institute of Science. The second author is thankful to DST-INSPIRE [DST/INSPIRE/04/2019/001914] for the financial support. The third author is supported by  J. C. Bose Fellowship from the Department of Science and Technology, Govt. of India.

\end{document}